\documentclass{amsart}

\usepackage{color}
\usepackage{hyperref}
\hypersetup{
    colorlinks=true,
    linkcolor=blue,
    filecolor=blue,      
    urlcolor=blue,
    citecolor=cyan,
}
\usepackage{amsthm}
\usepackage{extarrows}
\usepackage{amssymb}
\usepackage{mathrsfs}
\usepackage{amsmath}
\usepackage{graphicx}

\newtheorem{thm}{Theorem}
\newtheorem{lemm}[thm]{Lemma}
\newtheorem{prop}[thm]{Proposition}
\newtheorem{cor}[thm]{Corollary}

\theoremstyle{definition}

\newcommand{\dist}{\operatorname{dist}}
\newcommand{\norm}[1]{\|#1\|}

\title{The Hausdorff dimension of the spectrum of the Thue-Morse Hamiltonian}
\author{Liu Qinghui}
\address{School of Mathematics and Statistics, Beijing Institute of Technology, Beijing, China}
\email{qhliu@bit.edu.cn}
\thanks{Supported by National Natural Science Foundation of China (12571095).}

\author{Lv Yiran}
\address{School of Mathematics and Statistics, Beijing Institute of Technology, Beijing, China}
\email{3120256207@bit.edu.cn}
\date{}

\begin{document}

\begin{abstract}
In 2015, Liu, Qu (Commun. Math. Phys., Vol. 338) show that the Hausdorff dimension of the spectrum of the Thue–Morse Hamiltonian admits a positive lower bound independent of the coupling constant $\lambda$:
$$\dim_H \sigma(H_{\mathrm{tm},\lambda})\ge\frac{\log 2}{140\log 2.1}.$$
In this paper, we prove that the spectrum of the Thue--Morse Hamiltonian has full Hausdorff dimension for any $\lambda\in\mathbb{R}$:
\[
\dim_H \sigma(H_{\mathrm{tm},\lambda})=1.
\]
\end{abstract}
\maketitle

\section{Introduction}

\subsection{Background and Main Result}

Discrete Schr\"odinger operators with potentials generated by
substitutions have been studied extensively since the 1980s.
Transfer matrices and the associated trace maps provide an important
approach to their spectral analysis. 
In particular, the trace-map method, 
together with other approaches that can attain the same objective, 
has been used to study the zero-measure property of 
the spectrum for broad classes of substitutional potentials.; 
see, for example, \cite{BG93,LTWW02, Lenz02, DL06}. 
For such operators, once the spectrum is known to
have zero Lebesgue measure, a natural problem is to investigate its
finer geometric structure. The Hausdorff and box dimensions of the
spectrum have therefore become important objects in the spectral theory
of substitutional Hamiltonians. Among the most studied models in this
direction are the Fibonacci, Thue--Morse and period-doubling
Hamiltonians.

The Fibonacci Hamiltonian is the best understood among these models.
Its spectrum is a Cantor set of zero Lebesgue measure for every nonzero
coupling \cite{Suto89}, and its spectral structure is closely related
to the dynamics of the Fibonacci trace map. The hyperbolic structure of
the trace dynamics has made it possible to obtain rather detailed
information on the geometry of the spectrum; see, for example,
\cite{Cas86,DG09}. In particular, the dependence of the Hausdorff
dimension on the coupling constant is well understood in both the weak-
and strong-coupling regimes. Damanik and Gorodetski proved that\cite{DG11}
\[
    \lim_{\lambda\to0}
    \dim_H \sigma(H_{\mathrm{Fib},\lambda})
    =1,
\]
whereas Damanik, Embree, Gorodetski and Tcheremchantsev proved that\cite{DEGT08}
\[
    \lim_{\lambda\to\infty}
    (\log\lambda)
    \dim_H \sigma(H_{\mathrm{Fib},\lambda})
    =
    \log(1+\sqrt{2}).
\]
Thus, in particular,
\[
    \dim_H \sigma(H_{\mathrm{Fib},\lambda})
    \rightarrow 0,
    \quad\text{as }\lambda\to\infty.
\]
More generally, the fractal dimensions of spectra for Sturmian
Hamiltonians have been studied in \cite{LW04,LPW07,LQW14}.

A rather different phenomenon was discovered for the Thue--Morse
Hamiltonian. In \cite{LQ15}, Liu and Qu proved that, for every
\(\lambda\ne0\),
\[
    \dim_H \sigma(H_{\mathrm{tm},\lambda})
    \geq
    \frac{\log 2}{140\log 2.1}.
\]
The remarkable point here is that the lower bound is independent of the
coupling constant. Hence, in contrast with the Fibonacci Hamiltonian,
the Hausdorff dimension of the Thue--Morse spectrum remains uniformly
bounded away from zero as \(\lambda\to\infty\). The proof in
\cite{LQ15} was based on the dynamics of the Thue--Morse trace
polynomials. A crucial observation was that, after suitable rescaling
near certain regular germs, the iterated trace polynomials become
increasingly close to the model family
\[
    \{2\cos(2^n x)\}_{n\geq1}.
\]
This allowed the construction of a Cantor subset of the spectrum with a
Hausdorff dimension bounded from below uniformly in \(\lambda\).
Thus the coupling-independent estimate in \cite{LQ15} already suggested
that the Thue--Morse Hamiltonian possesses a dimensional mechanism quite
different from that of the Fibonacci model.

The study of spectral dimensions for substitutional Hamiltonians was
subsequently continued for other substitutions. For the period-doubling
Hamiltonian, Liu, Qu and Yao constructed an intrinsic coding of the
spectrum and a separating nested structure, and obtained
\[
    \dim_H \sigma(H_{\mathrm{pd},\lambda})
    >
    \frac{\log\frac{1+\sqrt{5}}{2}}{\log 4},
    \qquad \lambda>0,
\]
see \cite{LQY22}. This lower bound is again independent of \(\lambda\).
More recently, Liu and Tang studied a class of generalized Thue--Morse
Hamiltonians and obtained coupling-independent positive lower bounds
for their spectral Hausdorff dimensions \cite{LT23}. The constructions
in these works are different: the Thue--Morse argument in \cite{LQ15}
is based on regular germs, whereas the period-doubling and generalized
Thue--Morse arguments use suitable nested structures of periodic
approximating spectra. Nevertheless, all of them rely essentially on
the detailed dynamics of the corresponding trace maps.

We now return to the Thue--Morse Hamiltonian, which is the object of
the present paper. 
Let $\tau:\{a,b\}^{*}\longrightarrow\{a,b\}^{*}$
be the Thue--Morse substitution defined by
\[
    \tau(a)=ab,
    \qquad
    \tau(b)=ba.
\]
Starting from the letter \(a\), successive iterations give
\[
    \tau(a)=ab,\qquad
    \tau^{2}(a)=abba,\qquad
    \tau^{3}(a)=abbabaab,\quad\ldots
\]
Since \(\tau^{2}(a)\) begins and ends with \(a\), the pointed words
$ \tau^{2k}(a)\boldsymbol{\cdot}\tau^{2k}(a)$
are compatible on increasingly large neighborhoods of the origin.
Here the dot separates the coordinates \(-1\) and \(0\). Their limit
defines a two-sided Thue--Morse sequence
\[
    w_{\mathrm{tm}}
    =
    \lim_{k\to\infty}
    \tau^{2k}(a)\boldsymbol{\cdot}\tau^{2k}(a)
    \in\{a,b\}^{\mathbb Z}.
\]

Define, for $n\in\mathbb{Z}$,           
\[
    v_{\lambda}(n)
    =
    \begin{cases}
        \lambda, & w_{\mathrm{tm}}(n)=a,\\
        -\lambda, & w_{\mathrm{tm}}(n)=b.
    \end{cases}
\]
The corresponding Thue--Morse Hamiltonian
\(H_{\mathrm{tm},\lambda}\) acts on \(\ell^{2}(\mathbb Z)\) by
\[
    (H_{\mathrm{tm},\lambda}\psi)(n)
    =
    \psi(n+1)+\psi(n-1)+v_{\lambda}(n)\psi(n).
\]

We recall the following facts from \cite{LQ15}.

For \(n\geq1\), let \(x_n(E)\) denote the trace polynomial associated
with the word \(\tau^n(a)\) and 
\[
    \sigma_n
    :=
    \{E\in\mathbb R:|x_n(E)|\leq2\}.
\]
The Thue--Morse trace polynomials satisfy
the recursion
\begin{equation}
    x_{n+1}(E)
    =
    x_{n-1}^{2}(E)\bigl(x_n(E)-2\bigr)+2,
    \qquad n\geq2,
    \label{eq:TM-trace-recursion}
\end{equation}
with $x_1(t)=t^2-\lambda^2-2$ and $x_2(t)=(t^2-\lambda^2)^2-4t^2+2$,
where $\lambda$ is the coupling.
Since $\deg x_n(t)=2^n$, by Floquet theory, $\sigma_n$ consists of $2^n$ intervals,
which are often called bands of level $n$.
On each of these bands, $x_n(t)$ is strictly monotone and maps the band onto $[-2,2]$.
Two neighbor bands may intersect at one endpoint.
For $n\ge1$, 
$$\sigma_n\cup\sigma_{n+1} \subset  \sigma_{n-1}\cup\sigma_n,$$ 
and
\begin{equation}
    \sigma(H_{\mathrm{tm},\lambda})
    =
    \bigcap_{n\geq1}
    \bigl(\sigma_n\cup\sigma_{n+1}\bigr)=\limsup_{n\to\infty}\sigma_n.
    \label{eq:TM-periodic-approximation}
\end{equation}

Thus the spectral analysis is naturally connected with the trace dynamics generated by \eqref{eq:TM-trace-recursion}. 
In contrast to the Fibonacci case,
the Thue--Morse trace dynamics also has a rather different dynamical
structure from the hyperbolic trace-map picture arising in the
Fibonacci model. 
The trace recursion is equivalently written as the two-dimensional polynomial map (see \cite{AP89,B89})
\[
    \Phi(x,y)
    =
    (1-y,\,4x^2y),
\]
which is everywhere non-conformal. 
Nor is $\Phi$ hyperbolic. It possesses a center manifold at certain non-hyperbolic points.
This center-type structure, together with the non-conformality of the
trace map, is one of the principal dynamical features of the
Thue--Morse problem.

On the other hand, the study in \cite{LQ15,LQY17} is  heavily
based on dynamics near non-hyperbolic points;
and our study in this paper is based on dynamics on the corresponding center manifold,
which is an expanding Chebyshev dynamics.

We are now ready to state the main result of this paper.

\begin{thm}\label{thm:main}
For every real \(\lambda\), the spectrum of the Thue--Morse Hamiltonian
has full Hausdorff dimension:
\[
    \dim_H \sigma(H_{\mathrm{tm},\lambda})=1.
\]
\end{thm}

Theorem~\ref{thm:main} completely determines the Hausdorff dimension of
the Thue--Morse spectrum for all positive couplings. Recall that the
spectrum has zero Lebesgue measure. Thus the Thue--Morse spectrum is a
zero-measure spectral set whose Hausdorff dimension is nevertheless
maximal for every nonzero coupling \(\lambda\).

\paragraph{\textbf{Remark.}}
Theorem~\ref{thm:main} strengthens the coupling-independent phenomenon
discovered in \cite{LQ15}. There it was proved that
\[
    \dim_H \sigma(H_{\mathrm{tm},\lambda})
    \geq
    \frac{\log 2}{140\log 2.1},
\]
with a lower bound independent of \(\lambda\).
Theorem~\ref{thm:main} shows that this phenomenon is in fact maximal:
the Hausdorff dimension is not only uniformly bounded away from zero,
but is identically equal to \(1\) for every real coupling.

Since for $\lambda=0$, $\sigma(H_{\mathrm{tm},\lambda})=[-2,2]$, we only discuss the case of $\lambda\ne0$.

\subsection{Sketch of the main ideas}

We explain the main ideas of the proof. The starting point of our
argument is a new reformulation of the Thue--Morse trace recursion.
Define, for $n\ge2$,
\[
    y_n:=x_n-x_{n-1}^2+2,
\]
and then \eqref{eq:TM-trace-recursion} is equivalently written as
\begin{equation}
    \begin{cases}
        x_{n+1}=x_n^2-2+y_{n+1},\\[1mm]
        y_{n+1}=(2-x_n)y_n.
    \end{cases}
    \label{eq:TM-new-recursion}
\end{equation}

This simple reformulation isolates the underlying Chebyshev dynamics from the transverse deviation, allowing the two components to be analyzed separately in the subsequent proof.
When \(y_n=0\), the trace dynamics reduces exactly to the Chebyshev map
\[
    f(x)=x^2-2,
\]
for which
\[
    f(2\cos\theta)=2\cos2\theta.
\]
Through the angular parametrization \(x=2\cos\theta\), this dynamics is
semiconjugate to the doubling dynamics. 
For $\lambda=0$, the transverse deviation vanishes identically.
Thus, as long as the deviation
\(y_n\) is sufficiently small, a finite piece of the true trace orbit
can be regarded as a perturbation of the Chebyshev dynamics. 
The main problem is therefore to identify intervals on which 
the transverse deviation remains small after a long block of iterations.
By \eqref{eq:TM-new-recursion}, the transverse multiplier over \(n\)
steps is
\[
    \prod_{j=0}^{n-1}
    \bigl(2-f^j(2\cos\theta)\bigr).
\]
Since $f^j(2\cos\theta)=2\cos(2^j\theta)$, 
we have $2-f^j(2\cos\theta) =  4\sin^2(2^{j-1}\theta).$
Hence the logarithm of the transverse multiplier can be represented as
a Birkhoff sum for the doubling map. More precisely, let
\[
    T(x)=2x\pmod1,
    \qquad
    \varphi(x)=\log\bigl(4\sin^2\pi x\bigr),
\]
and write
\[
    S_n\varphi(x)
    =
    \sum_{j=0}^{n-1}\varphi(T^j x).
\]
After the corresponding angular normalization, the transverse growth
along the ideal orbit is governed by \(S_n\varphi\). Notice that
\[
    \int_0^1\varphi(x)\,dx=0.
\]
Thus there is no exponential contraction or expansion in the
transverse direction at the level of the Lyapunov exponent. The size
of the deviation is instead determined by the fluctuations of the
Birkhoff sums.

This is where a probabilistic ingredient enters the proof. A limit
theorem of Fernando and Schindler \cite{FS26} applies to the unbounded
observable \(\varphi\) for the doubling map and yields a central limit
theorem for \(S_n\varphi\). Combining this limit theorem with a uniform
oscillation estimate on dyadic cylinders, we show that, for all
sufficiently large \(K\), a positive proportion of the \(2^K\) dyadic
intervals of level \(K\) satisfy a uniform negative estimate for the
Birkhoff sum. More precisely, letting
\[
    I_{K,j}
    =
    \left[
        \frac{j}{2^K},
        \frac{j+1}{2^K}
    \right),
    \qquad
    0\leq j<2^K,
\]
one can select a family of such intervals that
\[
    S_K\varphi(t)\leq-\log20
\]
throughout every selected interval. On these intervals the transverse
multiplier is therefore uniformly smaller than a fixed constant less
than one.

The next step is to transfer this statistical statement from the ideal
Chebyshev dynamics to the true Thue--Morse trace dynamics. 
We start with a sufficiently deep spectral band \(I\) on which the deviation is very small. 
For the fixed large \(K\) above, 
by a Floquet argument,  
$I\cap\sigma_{n+1}\cap ... \cap \sigma_{n+K}$ is composed of $2^K$ child intervals.
Moreover, finite-time perturbation estimates
show that the true trace orbit shadows the ideal Chebyshev orbit during
the next \(K\) steps. 
These \(2^K\) child intervals correspond precisely to the  \(2^K\) 
branches of the ideal dynamics.

Among these child intervals, a fixed positive proportion can be retained
so that two properties hold simultaneously. First, the deviation at
the end of the \(K\)-step block is again sufficiently small. Second,
every retained child \(J\) satisfies a uniform lower bound
\[
    \frac{|J|}{|I|}
    \geq c\,2^{-K},
\]
where \(c>0\) does not depend on \(K\) or on the generation of the
construction. Thus the statistical abundance of good orbits for the
doubling map is converted into an abundance of sufficiently large
spectral descendants.

The crucial point is that the same configuration is recovered on every retained child. 
The deviation is again sufficiently small to repeat the same \(K\)-step construction. 
Hence, for every sufficiently large \(K\), we obtain a nested Cantor set
\[
    \Lambda_K
    \subset
    \sigma(H_{\mathrm{tm},\lambda})
\]
such that every parent interval has at least
\[
    N_K\geq c_0\,2^K
\]
retained children, while every retained child \(J\) satisfies
\[
    \frac{|J|}{|I|}
    \geq c_1\,2^{-K},
\]
where \(c_0,c_1>0\) are independent of \(K\) and of the generation.

The standard dimension estimate for such a nested construction gives
\[
    \dim_H\Lambda_K
    \geq
    \frac{\log N_K}
         {K\log2-\log c_1}\ge \frac{K\log2+\log c_0}{K\log2-\log c_1}.
\]
Letting $K\to\infty$, since $\Lambda_K\subset\sigma(H_{\mathrm{tm},\lambda})$,
Theorem~\ref{thm:main} follows.

\subsection{Some remarks}
Let us compare the present argument with the earlier
construction in \cite{LQ15}. The regular-germ method developed there
already identified the local cosine model and produced a Cantor subset
whose Hausdorff dimension is bounded uniformly away from zero for all
positive couplings. 
In contrast to previous studies, which were confined to neighborhoods of individual points, 
the new recursion \eqref{eq:TM-new-recursion} in the present work 
allows us to isolate the Chebyshev dynamics from the transverse deviation on an entire interval.
The fluctuations of
this deviation are then described by Birkhoff sums for the doubling
map. The resulting statistical information allows us to retain a fixed
positive proportion of all the \(2^K\) available children at each
 step. Thus the loss in the logarithm of the
branching number is only \(O(1)\), whereas both the branching and the
geometric contraction have leading order \(K\log2\). This almost
lossless branching mechanism is what upgrades the uniform positive
lower bound of \cite{LQ15} to full Hausdorff dimension.

The research in this paper indicates that, for dynamical system \eqref{eq:TM-new-recursion}, 
when $\lambda$ is large--that is, when the global perturbation is large--the regions 
of infinitesimal perturbation exist extensively. 
This is somewhat analogous to turbulence.

We have also studied Schr\"odinger operators with other substitution sequences as potentials. 
We find that for the generalized Thue–Morse potential(\cite{LT23}), 
the spectral dimension can be shown to be equal to $1$ 
by methods similar to those in this paper. 
For the period-doubling potential, the situation is somewhat different; 
however, we still believe that with some suitable adjustments, 
the conclusion of spectral dimension equal to $1$ can be obtained.

\subsection{Structure of the paper}
The rest of this paper is organized as follows. In Sect.~2, starting from a band \(I_0\) with small deviation, we obtain a class of cut-off intervals \(\mathcal{I}_{0,K}^{(1)}(I_0)\). In Sect.~3, we select from \(\mathcal{I}_{0,K}^{(1)}(I_0)\) those cut-off intervals with small deviation and collect them into \(\mathcal{I}_{0,K}^{(2)}(I_0)\). In Sect.~4, we choose from \(\mathcal{I}_{0,K}^{(2)}(I_0)\) those cut-off intervals with sufficiently large length and form \(\mathcal{I}_{0,K}^{(3)}(I_0)\). In Sect.~5, we construct a separated nested structure and estimate the Hausdorff dimension.

\section{From bands to cut-off intervals}

Setting $f(x)=x^2-2$, $g(x)=2-x$, we have
\begin{equation}\label{itx-xy}
 x_{n+1}-f(x_n)=g(x_n)(x_{n}-f(x_{n-1}))=y_{n+1},
\end{equation}
and hence \eqref{eq:TM-new-recursion} can be written as, for  $n\ge2$,
\begin{equation}\label{itxy-xy}
\begin{cases}
x_{n+1}=f(x_n)+y_{n+1},\\
y_{n+1}=g(x_n)y_n.
\end{cases}
\end{equation}

We restate some of the properties from \cite{LQ15} for convenience, and include a short proof for completeness. 

\begin{prop}\label{basic}
Take $t\in\mathbb{R}$.

(i) If $x_n(t)=2$ for some $n\ge2$, then for all $m\ge n$, $x_m(t)=2$.

(ii) If $x_n(t)=0$ for some $n\ge1$, then $x_{n+2}(t)=2$. 
Hence, for all $m\ge n+2$, $x_m(t)=2$.
Moreover, if $x_n(t)\ne2$, then for any $2\le k\le n$, $x_k(t)\ne2$,
and for any $1\le k\le n-2$, $x_k(t)\ne0$.

(iii) If $x_n(t)\le2$, then for any $k\ge1$, $x_{n+k}(t)\le2$.

(iv) $x_2(t)-f(x_1(t))=-4\lambda^2< 0$.

(v) If $n\ge4$ is even, $x_{n-1}(t)\ne2$, then $x_n(t)-f(x_{n-1}(t))<0$.

(vi) If $n\ge3$ is odd,  $x_{n-1}(t)<2$, then $x_n(t)-f(x_{n-1}(t))<0$.
\end{prop}

\begin{proof}
Properties (i), (ii), and (iii) are direct consequences of \eqref{eq:TM-trace-recursion}.
We get (iv) by a direct computation.

By \eqref{itx-xy} and \eqref{eq:TM-trace-recursion},
$$x_{n+2}(t)-f(x_{n+1}(t))=(2-x_n(t))^2x_{n-1}^2(t)(x_n(t)-f(x_{n-1}(t))).$$
If $n\ge4$ is even, then 
 $$x_n(t)-f(x_{n-1}(t))=(x_2(t)-f(x_1(t)))
 \prod_{i=1}^{n/2-1}(2-x_{2i}(t))^2 x_{2i-1}^2(t).$$
If $x_{n-1}(t)\ne 2$, then by (ii), $x_n(t)-f(x_{n-1}(t))<0$.

Suppose $n\ge3$ is odd. By (ii), (iv), (v),  and $x_{n-1}(t)\ne2$, we have
$$x_{n-1}(t)-f(x_{n-2}(t))<0.$$
If $x_{n-1}(t)<2$, then
$$x_{n}(t)-f(x_{n-1}(t))=(2-x_{n-1}(t))(x_{n-1}(t)-f(x_{n-2}(t)))<0.$$
\end{proof}

%\subsection{cutting-off intervals}

For every interval $I\subset \mathbb{R}$,
the $n$-level deviation on $I$ is defined to be
 $$d(n,I):=\sup_{t\in I}|y_{n}(t)|.$$

\begin{lemm}\label{smalldevi}
There is a decreasing sequence of bands $\{J_n\}_{n\ge0}$,
where $J_n$ is a band of level $n+k$ for all $n\ge1$ and some $k\ge2$,  such that
$$d(n+k,J_{n})<d(k,J_0)2^{-n(n-3)/2}.$$
Hence, for every $\delta>0$, there exists a positive integer $n$
and band $I_0$ of level $n$ such that $d(n,I_0)<\delta$.
\end{lemm}

\begin{proof}
Take a band $J_0=[a_0,b_0]$ of level $k$.
Without loss of generality, suppose $x_k(t)$ is strictly decreasing on $J_0$,
i.e., $x_k(a_0)=2$, $x_k(b_0)=-2$, and $x_k(J_0)=[-2,2]$.

For all $n>0$,
by Proposition \ref{basic} (i) (iii),  
we may take $J_{n}$ to be the unique band $[a_n,b_n]$ of level $n+k$ with $a_n=a_0$.
Note that $x_{n+k}$ is decreasing on $J_n$.

Take any $n>0$.
Since there exists a unique $c_{n-1}\in(a_{n-1},b_{n-1})$ such that $x_{n+k-1}(c_{n-1})=0$, 
by Proposition \ref{basic} (v) (vi), 
$$x_{n+k}(c_{n-1})<f(x_{n+k-1}(c_{n-1}))=-2.$$
This implies $b_{n}<c_{n-1}<b_{n-1}$, and hence $J_{n}\subset J_{n-1}$.
In particular, we have
$$%%\begin{equation}\label{bis}
a_0<b_n<b_{n-1}<\cdots<b_{1}<b_0.
$$%\end{equation}
Then for $0\le i<n$, $-2<x_{i+k}(b_n)<2$, and hence by Proposition \ref{basic} (v) (vi), 
\begin{equation}\label{negative}
x_{i+k+1}(b_n)-f(x_{i+k}(b_n))<0.
\end{equation}

Define $h(t)=\sqrt{t+2}$. We have $f\circ h=id$ when $t\ge-2$ and $h$ is strictly increasing for $t\ge-2$. 

Since, by \eqref{negative}, $f(x_{n+k-1}(b_n))>x_{n+k}(b_n)=-2$, we have 
$$x_{n+k-1}(b_n)=h(f(x_{n+k-1}(b_n)))>h(-2)=0.$$
By \eqref{negative}, $f(x_{n+k-2}(b_n))>x_{n+k-1}(b_n)>0$, we have
$$x_{n+k-2}(b_n)=h(f(x_{n+k-2}(b_n)))>h(0).$$
By induction, for $1\le j\le n$,
$$x_{n+k-j}(b_n)>h^{j-1}(0).$$
For $j\ge0$, since
\begin{equation*}
g(h^{j+1}(0))=2-\sqrt{h^j(0)+2}=2-\sqrt{4-g(h^j(0))}
    =\frac{g(h^j(0))}{2+\sqrt{4-g(h^j(0))}}\le\frac{g(h^j(0))}{2},
\end{equation*}
we have $0< g(h^j(0))\le 2^{-j+1}$.
Then, for any $t\in J_n$, for $1\le j\le n$,
\begin{equation*}
h^{j-1}(0)< x_{n+k-j}(t)\le 2,
\end{equation*}
and hence
$$0\le g(x_{n+k-j}(t))<2^{-j+2}.$$

Thus, for $t\in J_n$, a band of level $n+k$,
\begin{equation*}
|y_{n+k}(t)|=\prod_{j=1}^{n}g(x_{n+k-j}(t))|y_k(t)| <
2^{\sum_{j=1}^{n}(-j+2)}d(k,J_0)=d(k,J_0)2^{-n(n-3)/2}.
\end{equation*}
This completes the proof.
\end{proof}

Therefore, we may begin with a band $I_0$ of a large level $n$
 with $d(n,I_0)$ small enough.

\begin{lemm}
\label{lemma3}
For every integer $K\ge 1$, there exists $\delta_1>0$ such that, 
if $I_0$ is a band of level $n$  with $d(n,I_0)<\delta_1$, then  
\begin{equation*}
 I_0\cap\sigma_{n+1}\cap\cdots\cap\sigma_{n+K}
=I_1\sqcup\cdots\sqcup I_{2^K},
\end{equation*}
    where $I_1,\ldots,I_{2^K}$ are $2^K$ different intervals,
    and
\begin{equation*}
    x_{n+K}(I_j)\supset [-2,2\cos{\frac{\pi}{2^{K+1}}}]
\end{equation*}
    for  $ j=1,\ldots, 2^K$.
\end{lemm}

\begin{proof}
For $n\ge1$, $1\le j\le K$, by \eqref{itxy-xy}, 
$x_{n+j}$ is a polynomial in $x_{n},y_{n}$, which we denote by 
$$x_{n+j}=h_j(x_{n},y_{n}).$$
Take $\varepsilon=2^{-20K}$.
There exist $\delta_{1,j}>0$ such that, for any 
$x\in [-3,3]$, $|y|<\delta_{1,j}$,
\begin{equation}\label{pertb}
|h_j(x,y) - h_j(x,0)|<\varepsilon.
\end{equation}

Let $\eta=\min_{j=1,\ldots,K}\{\delta_{1,{j}}\}$ and $\delta_1=\eta/5^K$.
Suppose $I_0$ is a level $n$ band with $d(n,I_0)<\delta_1$.
Then $x_n$ is monotone on $I_0$ and $x_n(I_0)=[-2,2]$.
Without loss of generality, suppose $x_n$ is strictly decreasing on $I_0$.

Notice that $h_j(x,0)=f^j(x)$, $|f^j(x)|\le2$ if $|x|\le2$, and $|y_n|\le\delta_1$.
By \eqref{pertb}, for any $t\in I_0$, $1\le k\le K$,
$$|x_{n+k}(t)|<2+\varepsilon,\quad |y_{n+k}(t)|<5^k|y_n|\le\eta.$$
Since  $x_{n+k+j}=h_j(x_{n+k},y_{n+k})$, we have
\begin{equation}\label{err}
\sup\{|x_{n+k+j}(t)-f^j(x_{n+k}(t))| : k\ge0, j\ge1, 1\le j+k \le K,
t\in I_0 \}<\varepsilon.
\end{equation}

We parameterize $I_0$ by $\xi: [0,\pi]\to I_0$ 
so that $x_{n}(\xi(\theta))=2\cos{\theta}$.

Notice that $f^j(2\cos{\theta})=2\cos{2^j\theta}$,
one may check that, for $1\le j\le K$,
\begin{equation}\label{meas}
\mathcal{L}\{\theta\in[0,\pi]:x_{n+j}(\xi(\theta))\le -2+\varepsilon\}
<\pi\sqrt{\varepsilon},
\end{equation}
where $\mathcal{L}$ is the Lebesgue measure. 
Indeed, since $f^j(x_n(\xi(\theta)))=2\cos 2^j\theta$, 
$$\begin{array}{cl}
&\mathcal{L}\{\theta\in[0,\pi] : x_{n+j}(\xi(\theta))\le -2+\varepsilon\}\\
\le &\mathcal{L}\{\theta\in[0,\pi] : f^j(x_n(\xi(\theta)))< -2+2\varepsilon\}\\
= &\mathcal{L}\{\theta\in[0,\pi] : 2\cos 2^j\theta< -2+2\varepsilon\}\\
= &\mathcal{L}\{\theta\in[0,\pi] : 2\cos 2\theta< -2+2\varepsilon\}\\
=&2\arccos(1-\varepsilon)<\pi\sqrt{\varepsilon},
\end{array}$$
where the first inequality is due to \eqref{err};
the second equality follows from the fact that 
$\theta\to 2\theta\mod 2\pi$ preserves Lebesgue measure on $[0,2\pi]$;
the last inequality is due to, letting $\alpha=\arccos(1-\varepsilon)$, 
by $\sin x>{2\sqrt{2}x}/{\pi}$ for $0< x\le \pi/4$,
$$\varepsilon=1-\cos\alpha=2\sin^2\frac{\alpha}{2}
>\frac{4\alpha^2}{\pi^2}=\left(\frac{2\arccos(1-\varepsilon)}{\pi}\right)^2.$$

Denote $\delta_2=\pi\sqrt{\varepsilon}(\ll 2^{-K})$. 
We will show that, for $k=1,\cdots,2^K$,
$$I_k\supset[\xi(\frac{(k-1)\pi}{2^K}+\delta_2),\xi(\frac{k\pi}{2^K}-\delta_2)].$$

Take any $t\in I_0$. Since $x_n(t)\le2$, 
by Proposition \ref{basic} (iii), $\forall j\ge1$, $x_{n+j}(t)\le2$.

Suppose $x_m(t)=-2$ for some $n\le m< n+K$. Notice that $f(-2)=f(2)=2$.
By Proposition \ref{basic} (ii),(v),(vi), $x_{m+1}(t)<2.$
By \eqref{err}, for $1\le j\le n+K-m$,
\begin{equation}
    \label{>2-varepsilon}
x_{m+j}(t)>2-\varepsilon.
\end{equation}
And then by induction and Proposition \ref{basic} (ii), for $1\le j\le n+K-m$,
$x_{m+j}(t)<2.$
In summary, for $1\le j\le n+K-m$,
$$2-\varepsilon<x_{m+j}(t)<2.$$ 
In particular, since $x_n(\xi(\pi))=-2$, 
for $1\le j\le K$, $2-\varepsilon<x_{n+j}(\xi(\pi))<2$.

If $x_m(t)=0$ for some $n\le m< n+K$, then $x_{m+j}(t)=2$ for any $j>1$. 
Moreover, by Proposition \ref{basic} (v),(vi),
$$x_{m+1}(t)<f(x_m(t))=-2.$$

First, by definition, $x_n(\xi(\pi/2))=2\cos\pi/2=0$. 
Then $x_{n+1}(\xi(\pi/2))<-2$.
Since $$x_{n+1}(\xi(0))=2,\quad x_{n+1}(\xi(\pi))>2-\varepsilon,$$ 
by the intermediate value theorem, there exist $\theta^{-}_{1/2},\theta^{+}_{1/2}\in(0,\pi)$ such that
$$0<\theta^{-}_{1/2}<\pi/2<\theta^{+}_{1/2}<\pi,\quad
x_{n+1}(\xi(\theta^{-}_{1/2}))=x_{n+1}(\xi(\theta^{+}_{1/2}))=-2.$$
Note that there are at most $2^{n+1}$ solutions to the equation $x_{n+1}(t)=-2$.
There are $2^n$ bands in $\sigma_n$, $I_0$ is one of these bands.
As in $I_0$,  each of the $2^n$ bands contains  two distinct points solving the equation.
Then $\theta^{-}_{1/2},\theta^{+}_{1/2}\in(0,\pi)$ are unique.

Since $f(2\cos\theta)=2\cos2\theta$, $f(x_n(\xi(\pi/2)))=-2$. By \eqref{meas}, 
$$\theta^{-}_{1/2}, \theta^{+}_{1/2}\in(\pi/2-\delta_2,\pi/2+\delta_2).$$
And by \eqref{>2-varepsilon}, for $j=2,\cdots,K$,
$$x_{n+j}(\xi(\theta^{-}_{1/2})),x_{n+j}(\xi(\theta^{+}_{1/2}))\in(2-\varepsilon,2).$$
Note that $x_{n+1}$ is strictly monotone on $[\xi(0),\xi(\theta^{-}_{1/2})]$ 
and $[\xi(\theta^{+}_{1/2}),\xi(\pi)]$ and
$$[\xi(0),\xi(\theta^{-}_{1/2})]\cup[\xi(\theta^{+}_{1/2}),\xi(\pi)]
=I_0\backslash(\xi(\theta^{-}_{1/2}),\xi(\theta^{+}_{1/2}))
=I_0\cap \sigma_{n+1},$$
with $x_{n+1}([\xi(0),\xi(\theta^{-}_{1/2})])\supset [-2,2-\varepsilon]$,
$x_{n+1}([\xi(\theta^{+}_{1/2}),\xi(\pi)])\supset [-2,2-\varepsilon]$.

Next, we continue this process up to $n+2$ to obtain 
$\theta^{-}_{1/4}, \theta^{+}_{1/4}$ and $\theta^{-}_{3/4}, \theta^{+}_{3/4}$.
By definition,
 $x_{n+2}(\xi(\pi/4))<f^2(x_n(\xi(\pi/4)))+\varepsilon=-2+\varepsilon$.
 Then
 \begin{equation*}
  \theta^{-}_{1/4}, \theta^{+}_{1/4},\pi/4\in  \{\theta\in[0,\pi]:x_{n+2}(T(\theta))<-2+\varepsilon\}.
\end{equation*}
By \eqref{meas}, we have
\begin{equation*}
|\theta^{-}_{1/4}-\pi/4|<\delta_2, \quad |\theta^{+}_{1/4}-\pi/4|<\delta_2.
\end{equation*}
The same argument applies to $\theta^{-}_{3/4}$ and $\theta^{+}_{3/4}$.
Note also that $x_{n+2}$ is strictly monotone on $[\xi(0),\xi(\theta^{-}_{1/4})]$,
$[\xi(\theta^{+}_{1/4}),\xi(\theta^{-}_{1/2})]$,
$[\xi(\theta^{+}_{1/2}),\xi(\theta^{-}_{3/4})]$ and $[\xi(\theta^{+}_{3/4}),\xi(\pi)]$,
and
\begin{equation*}
\begin{aligned}
I_0\cap\sigma_{n+1}\cap\sigma_{n+2}
={}&[\xi(0),\xi(\theta^-_{1/4})]
   \cup[\xi(\theta^+_{1/4}),\xi(\theta^-_{1/2})]\\
 &{}\cup[\xi(\theta^+_{1/2}),\xi(\theta^-_{3/4})]
   \cup[\xi(\theta^+_{3/4}),\xi(\pi)].
\end{aligned}
\end{equation*}
Again by \eqref{>2-varepsilon}, the value of $x_{n+2}$ on
 each of the four intervals above contains $[-2,2-\varepsilon]$.

 Continuing this process up to level $n+K$,
 we obtain the edge points of gaps,
 \begin{equation*}
    \Theta^{+}:=\{\theta^{+}_{k/2^K}:k=1,\ldots,2^K-1\},\quad
    \Theta^{-}:=\{\theta^{-}_{k/2^K}:k=1,\ldots,2^K-1\}.
 \end{equation*}
 Here the subscript denotes the dyadic rational as a real number, not as a formal fraction.
Thus by construction, for $k=1,\ldots,2^K-1$,
\begin{equation*}
\theta^{+}_{k/2^K},\theta^{-}_{k/2^K}\in [\frac{k\pi}{2^K}-\delta_2,\frac{k\pi}{2^K}+\delta_2].
\end{equation*}
Furthermore, for $k$ odd, we have
$x_{n+K}(\xi(\theta^{-}_{k/2^K}))=x_{n+K}(\xi(\theta^{+}_{k/2^K}))=-2$,
since they are the endpoints of the gap of level $n+K$.
For $k$ even, 
$$2-\varepsilon<x_{n+K}(\xi(\theta^{-}_{k/2^K})),\ 
x_{n+K}(\xi(\theta^{+}_{k/2^K}))<2,$$ 
since they are the endpoints of gaps created at earlier levels.

Denote $\theta^{+}_{0}=0,\theta^{-}_{1}=\pi$.
Therefore, the required disjoint intervals
\begin{equation*}
    I_k=[\xi(\theta^{+}_{(k-1)/2^K}),\xi(\theta^{-}_{k/2^K})],\quad k=1,\ldots,2^K,
\end{equation*}
are intervals on which $x_{n+K}$ is strictly monotone, and
\begin{equation*}
x_{n+K}(I_k)\supset [-2,2-\varepsilon],\quad k=1,\ldots,2^K.
\end{equation*}
Moreover, $ I_0\cap\sigma_{n+1}\cap\cdots\cap\sigma_{n+K}
=I_1\sqcup\cdots\sqcup I_{2^K}$.
\end{proof}

Recall that we call an interval $I$ a band of level $n$,
if $I$ is a connected component of $\sigma_n$, i.e.,
$x_n$ is monotone on $I$ and $x_n(I)=[-2,2]$.
Now, we call an interval $I$ a cut-off interval of level $n$ for a fixed $K>0$, 
or simply a cut-off interval,
if $I$ is a sub-interval of a  band of level $n$ and
$$x_n(I)\supset[-2,2\cos\frac{\pi}{2^{K+1}}].$$

We denote the class of $2^K$ intervals obtained in Lemma \ref{lemma3} by
$$\mathcal{I}_{0,K}^{(1)}(I_0):=\{I_1, I_2,\cdots I_{2^K}\}.$$
All these intervals are cut-off intervals of level $n+K$ for the given $K$, 
since by definition of $\varepsilon$, $2-\varepsilon>2\cos(\pi/2^{K+1}).$

Lemma \ref{lemma3} also applies to the cut-off intervals.
\begin{cor}
\label{lemma3'}
Let $K\ge1$ and $\delta_1=\delta_1(K)$ be the constant given by Lemma \ref{lemma3}. 
If $I'$ is a cut-off  interval of level $n$ for the given $K$, and $d(n,I')<\delta_1$, then
\begin{equation*}
I'\cap\sigma_{n+1}\cap\cdots\cap\sigma_{n+K}
\supset I_2'\sqcup\cdots\sqcup I_{2^K-1}',
\end{equation*}
    where $I_2',\ldots,I_{2^K-1}'$ are $2^{K}-2$ cut-off intervals of level $n+K$ for $K$.
\end{cor}

\begin{proof}
The proof is the same as the proof of Lemma \ref{lemma3}, except that
we need to deal with $[\xi({\pi}/{2^K}),\xi(\pi)]\subset I'$. 
We get, for $j=2,\cdots 2^{K}-1$,
\begin{equation*}
I_j'=[\xi(\theta^{+}_{(j-1)/2^K}),\xi(\theta^{-}_{j/2^K})].
\end{equation*}
Notice that $[\xi(0),\xi(\theta^{-}_{1/2^K})]$ is not entirely contained in $I'$.
\end{proof}
For a cut-off interval $I'$, we denote the family 
$\{I_2',\ldots,I_{2^K-1}'\}$ by $\mathcal I_{0,K}^{(1)}(I')$.

\noindent\textbf{Remark.} The indices are designed to be invariant under $j\to 2^{K}+1-j$.
This symmetry will be used in the proof of Lemma \ref{length}.

\section{cut-off intervals with small deviation}
%%%%%%%%%%%%新lem3

We need to check whether the cut-off intervals obtained in the previous section have small deviation.
Notice that 
\begin{equation} 
\prod_{j=0}^{n-1}g(f^j(2\cos{\theta}))
=\prod_{j=0}^{n-1} g(2\cos{2^j\theta})
=4^n\prod_{j=0}^{n-1} \sin^2{2^{j-1}\theta}.
\end{equation}
Set $R_n(\theta)=4^n\prod_{j=0}^{n-1} \sin^2{2^{j-1}\theta}$.
Now we analyse the product.
Let $I=[0,1)$,  $T:I\to I$ be the doubling map $T(x)=2x \mod 1$. 
$T$ is strong mixing with respect to the Lebesgue measure $\mathcal{L}$. 
For $x\in I$, define
$$\varphi(x)=\log 4\sin^2\pi x.$$
Define the Birkhoff sum of $\varphi$ as 
$$S_n \varphi(x) = \sum_{k=0}^{n-1} \varphi(T^k x).$$

We will prove the following two lemmas later.

\begin{lemm}\label{clt}
The central limit theorem holds for $(T,\mathcal{L},\varphi)$, i.e., 
$$\frac{S_n \varphi}{\sigma\sqrt{n}}\overset{d}\longrightarrow \mathscr{N}(0,1)$$
for some fixed $\sigma>0$.
\end{lemm}

Given $n>1$, the dyadic intervals of length $2^{-n}$ are
\[
I_{n,j}=\bigl[j\,2^{-n},\,(j+1)2^{-n}\bigr),\qquad j=0,1,\dots,2^n-1.
\]

\begin{lemm}\label{discrete}
Let
$$Y_n=\{I_{n,j} : 0\le j<2^n \mbox{ and } S_n \varphi(t)\le-\log 80,\; \forall t\in I_{n,j} \}.$$
We have
$$\liminf_{n\to\infty}\frac{\sharp Y_n}{2^n}>\frac{1}{5}.$$
\end{lemm}
\paragraph{\textbf{Remark.}} Notice that $S_n\varphi(x)=S_n\varphi(1-x)$. 
Thus $Y_n$ is symmetric in the sense that
 $I_{n,j}\in Y_n$ whenever $I_{n,2^n-1-j}\in Y_n$.

 By the two lemmas, we can find a large number of cut-off intervals with small deviation.
 
 \begin{cor}
\label{1924}
(1) There exists an integer $K_1>100$, such that for every $K>K_1$, 
\begin{equation}
\label{1/20}
\#\{I_{K,j}:\sup_{\theta\in I_{K,j}}R_K(\pi\theta)<1/20, 0\le j<2^K\}>2^K/6.
\end{equation}
(2) For $K>K_1$, there exists $\delta_3>0$, such that if $d(n,I_0)<\min\{\delta_1,\delta_3\}$
for a cut-off interval $I_0$ of level $n$, then there exists a symmetric family
\begin{equation*}
\mathcal{I}_{0,K}^{(2)}(I_0)\subset\{I\in\mathcal{I}_{0,K}^{(1)}(I_0):
d(n+K,I)<\frac{1}{10}d(n,I_0)\},
\end{equation*}
such that $\#\mathcal{I}_{0,K}^{(2)}(I_0)>2^K/6$.
\end{cor}

\begin{proof}
(1) Notice that 
\begin{equation*}
R_K(\pi\theta)=4\sin^2{(\pi\theta/2)}\mathrm{exp}(S_{K-1}\varphi(\theta))\le 4\mathrm{exp}(S_{K-1}\varphi(\theta)).
\end{equation*}
Then Lemma \ref{discrete} gives the proof.

(2) In the proof of Lemma \ref{lemma3},  
every $I_j\in \mathcal{I}_{0,K}^{(1)}(I_0)$ is a subset of $\xi(\pi I_{K,j-1})$.
     Furthermore, 
\begin{equation*}
    y_{n+K}= g(x_{n+K-1})\cdots g(x_{n+1})g(x_n)y_n.
    \end{equation*}
As in the proof of Lemma \ref{lemma3}, the function $g(x_{n+K-1})\cdots g(x_{n+1})g(x_n)$ is also
a polynomial in $x_n, y_n$, which we denote by $G(x_n,y_n)$. 
Thus there exists $\delta_3>0$ such that, for any $x\in [-2,2]$, 
for any $y$ with $|y|<\delta_{3}$, 
\begin{equation*}
|G(x,y)-G(x,0)|<1/20.
\end{equation*}
For $t\in I_0$, there is a $\theta\in[0,1]$ such that $x_n(t)=2\cos\pi\theta$.  By definition of $R_K$,
%Now take some $k\in\{0,2^K-1\}$ such that $P_k$ holds the condition \eqref{1/20}. We have
\begin{equation*}
G(x_n(t),0)=g(f^{K-1}(x_n(t)))\cdots g(f(x_n(t)))g(x_n(t))=R_{K}(\pi\theta).
\end{equation*}
If $R_K(\pi\theta)<1/20$, then
\begin{equation*}
|y_{n+K}(t)|=G(x_n(t),y_n(t))|y_n(t)|\le d(n,I_0)/10.
\end{equation*}
Define 
$$\mathcal{I}_{0,K}^{(2)}(I_0)=\{I_j\in\mathcal{I}_{0,K}^{(1)}(I_0) : 
I_{K,j-1}\in Y_n, j=1,\cdots, 2^K\}.$$
Hence, the fact that $\mathcal{I}_{0,K}^{(2)}(I_0)$ is symmetric and $\#\mathcal{I}_{0,K}^{(2)}(I_0)>2^K/6$ follows directly from remark of Lemma \ref{discrete} and  \eqref{1/20}.
\end{proof}

\subsection{Proof of Lemma \ref{clt}}
In \cite{FS26}, let $I=[0,1]$ and let $T: I\to I$ satisfy the full branch assumption 
(for instance the doubling map). 
Fernando and Schindler prove the following theorem.

\begin{thm}[\cite{FS26}, Theorem 2.5]\label{thm1}
Suppose $h$ is continuous and the right and left derivatives of $h$ exist on $I^o$. 
Suppose further that $h$ is not coboundary and there exist constants $a, b> 0$ such that
\begin{equation}\label{c1}
|h(x)| \lesssim x^{-a}(1-x)^{-a},\quad
\max\{|h'(x+)|, |h'(x-)|\} \lesssim x^{-b}(1-x)^{-b}.
\end{equation}
Assume
\begin{equation}\label{c2}
a < \min\{\vartheta, 1/b, 1/2\} \cdot \min\{1, \log \eta_- / \log \eta_+\},
\end{equation}
where $\vartheta, \eta_-, \eta_+$ are  determined by $T$. 
Then the following Central Limit Theorem holds:
$$\frac{S_n h - n \mathcal{L}(h)}{\sigma\sqrt{n}}\xrightarrow{d} \mathscr{N}(0,1).$$
\end{thm}
Here and in what follows, we denote by $I^o$ the interior of interval $I$.
%% **步骤1：验证增长条件**

The function \(\varphi(x)\) is singular at \(x=0\) and \(x=1\).
Since $\sin(\pi x) \sim \pi x$ as $x\to0$ 
and $\sin(\pi x) \sim \pi (1-x)$ as \(x\to 1\), we have
$$\begin{array}{l}
\varphi(x) = \log(4\sin^2\pi x) \sim 2\log x + \text{constant}, \quad x \to 0^+\\
\varphi(x) = \log(4\sin^2\pi x) \sim 2\log(1-x) + \text{constant}, \quad x \to 1^-.
\end{array}$$
Hence, for any small \(a>0\), 
$$|\varphi(x)| \lesssim x^{-a}(1-x)^{-a}.$$

%% **步骤2：验证导数条件**
A direct computation gives 
$$\varphi'(x) = \frac{8\pi \sin(\pi x)\cos(\pi x)}{4\sin^2(\pi x)} = 2\pi \cot(\pi x).$$
We have
$$
\varphi'(x) \sim \frac{2}{x}, \quad x \to 0^+; \quad \varphi'(x) \sim -\frac{2}{1-x}, \quad x \to 1^-.
$$
Let \(b=1\). Then we have
$$\max\{|\varphi'(x+)|, |\varphi'(x-)|\} \lesssim x^{-b}(1-x)^{-b}.$$

Hence,  \eqref{c1} is satisfied.

$T$ can be written as $\psi_1:[0,1/2)\to I$, $\psi_1(x)=2x$, and $\psi_2:[1/2,1)\to I$, $\psi_2(x)=2x-1$.
By definition in section 2.2 of \cite{FS26}, $\eta_+=\eta_-=2$, $\vartheta=1$, 
where $\vartheta$ is the H\"{o}lder exponent of $\psi_j^{-1}$.
Since $a>0$ can be arbitrarily small, \eqref{c2} is satisfied.

%% **步骤3：验证非余循环条件**

We claim that \(\varphi(x)=\log(4\sin^2\pi x)\) is not $T$-coboundary. 
We prove by contradiction. 
Suppose $\varphi$ is $T$-coboundary, i.e., there is \(\gamma\in L^2(I)\) 
such that \(\varphi = \gamma \circ T - \gamma\). 

Since  $2|\sin\pi t|=|1-e^{2i\pi t}|$,
$$\log 2|\sin\pi t|=Re(\log (1-e^{2i\pi t}))=
Re(-\sum_{l=1}^\infty\frac{1}{l}e^{2i\pi lt})=
-\sum_{l=1}^\infty\frac{1}{l}\cos 2\pi lt,$$
the Fourier series of $\varphi$ is given by
$$\varphi(t)=-2\sum_{l=1}^\infty\frac{\cos 2\pi l t}{l}
=-\sum_{l\ne0}\frac{e^{2i\pi lt}}{|l|}.$$
We see that $\int_0^1\varphi(t)\mathrm{d}t=0$. 

Let $\{\hat{\varphi}(k)\}_{k\in\mathbb{Z}}$, 
$\{\hat{\gamma}(k)\}_{k\in\mathbb{Z}}$ be 
the Fourier coefficients of $\varphi$, $\gamma$.
A direct calculation yields $\hat{\varphi}(k)=-1/|k|$ for $k\ne0$, 
and by \(\varphi = \gamma \circ T - \gamma\),
$$\hat{\varphi}(k)=\int_0^1\varphi(x)e^{-2\pi i k x}\mathrm{d}x=\int_0^1\gamma(2x\mod1)e^{-2\pi ikx}dx-\hat{\gamma}(k).$$
By setting $u=2x$,
$$\int_0^{1/2}\gamma(2x\mod1)e^{-2\pi ikx}dx
=\frac{1}{2}\int_0^1\gamma(u)e^{-2\pi iku/2}du$$
and by setting $u=2x-1$, 
$$\int_{1/2}^1\gamma(2x\mod1)e^{-2\pi ikx}dx
=\frac{e^{-\pi ik}}{2}\int_0^1\gamma(u)e^{-2\pi iku/2}du.$$
Hence, for odd $k$, 
$$\hat{\gamma}(k)=\left(\frac{1}{2}-\frac{1}{2}\right)
\int_0^1\gamma(u)e^{-2\pi iku/2}du-\hat{\varphi}(k)=-\hat{\varphi}(k)=1/|k|,$$
and for even $k=2l$,  
$$\hat{\varphi}(2l)=\left(\frac{1}{2}+\frac{1}{2}\right)
\int_0^1\gamma(u)e^{-2\pi i l u}du-\hat{\gamma}(2l)
=\hat{\gamma}(l)-\hat{\gamma}(2l),$$ 
which yields
$$\hat{\gamma}(2l)=\hat{\gamma}(l)-\hat{\varphi}(2l).$$
For an odd $l$, we have
$$\begin{array}{l}
\hat{\gamma}(2l)=\frac{1}{|l|}+\frac{1}{2|l|}=\frac{2^2-1}{2^1|l|}\\
\hat{\gamma}(4l)=\hat{\gamma}(2l)-\hat{\varphi}(4l)
=\frac{2^2-1}{2^1|l|}+\frac{1}{2^2|l|}
=\frac{2^3-1}{2^2|l|},\cdots\\
\hat{\gamma}(2^n l)=\frac{2^{n+1}-1}{2^n|l|}\to \frac{2}{|l|} (n\to\infty).
\end{array}$$
This contradicts the fact that $\gamma\in L^2(I)$.
Therefore, $\varphi$ is not a $T$-coboundary. 
All the assumptions of Theorem \ref{thm1} are now satisfied. Since
\[
\mathcal{L}(\varphi)=\int_0^1 \varphi(t)dt=0,
\]
Theorem \ref{thm1} yields
\[
\frac{S_n\varphi}{\sigma\sqrt n}
\overset{d}{\longrightarrow} \mathscr{N}(0,1)
\]
for some $\sigma>0$. This completes the proof.

\subsection{Proof of Lemma \ref{discrete}}

For $n>1$, define 
$$\Gamma_n:=\{t\in I : S_n\varphi(t)\le-\log^2 n-\log 80\}.$$
By Lemma \ref{clt},
\begin{equation}\label{cltmeas}
\lim_{n\to\infty} \mathcal{L}(\Gamma_n)=\frac{1}{2}.\end{equation}

Take $a=1/10$. Take large $n$ so that $\mathcal{L}(\Gamma_n)>\frac{1}{2}-a=2/5$.

For $j=0,1,\cdots, 2^n-1$, define
$$\tilde{I}_{n,j}:=[\frac{j+a}{2^{n}},\frac{j+1-a}{2^{n}}),\quad 
\tilde{t}_{n,j}:=\frac{j+1/2}{2^n}.$$
Let 
$$\Lambda_n:=\{0\le j\le2^n-1 : \Gamma_n\cap\tilde{I}_{n,j}\ne\emptyset\}.$$
For each $j\in\Lambda_n$, take a $t_{n,j}\in \Gamma_n\cap\tilde{I}_{n,j}$.
We have $\sharp \Lambda_n\ge2^n/4$. 
Otherwise, $\#\{j:\Gamma_n\cap\tilde{I}_{n,j}=\emptyset\}>\frac342^n$, thus
$$\mathcal{L}\left(\bigcup\{\tilde{I}_{n,j}:j\not\in\Lambda_n \}\right)
>\frac{3}{4}2^n\frac{1-2a}{2^n}=3/5,$$
which is a contradiction to \eqref{cltmeas}.

Let
\[
M_n = \lceil 6\log_2 n\rceil ,\qquad \eta_n = n^{-3}.
\]
Divide $S_n \varphi$ into two parts:
\[
S_n \varphi(x)=A(x)+B(x),
\]
where
\[
A(x)=\sum_{k=0}^{n-M_n-1} \varphi(T^k x),\qquad
B(x)=\sum_{k=n-M_n}^{n-1} \varphi(T^k x).
\]
(The number of terms in $A$ is $n-M_n$, in $B$ is $M_n$.)

All $O(\cdot)$ estimates are uniform in $n$ for $n$ sufficiently large.

\smallskip
\paragraph{\textbf{The first part $A$}}

We say that
an interval $I_{n,j}$ is \emph{good} (for the first part) if
\[
\Vert{T^k \tilde{t}_{n,j}}\Vert_{\mathbb{Z}} \ge 2\eta_n,\quad\text{for all } 0\le k\le n-M_n-1.
\]
Denote
\[
\mathcal G_n = \{ I_{n,j}\mid I_{n,j} \text{ is good}\},\qquad
\mathcal B_n = \{ I_{n,j}\mid I_{n,j}\notin \mathcal G_n\}.
\]

\begin{prop}\label{prop:bad_measure}
We have
\[
\mathcal{L}\left(\bigcup_{I_{n,j}\in\mathcal B_n} I_{n,j}\right) = O(n^{-2}).
\]
\end{prop}
\begin{proof}
For a fixed $k\le n-M_n-1$ consider the set
\[
E_k = \{ x\in[0,1) : \norm{T^k x}_{\mathbb{Z}} < 2\eta_n \}.
\]
Since $T$ preserves the Lebesgue measure $\mathcal{L}$,
\[
\mathcal{L}(E_k) = \mathcal{L}\{ y : \norm{y}_{\mathbb{Z}} < 2\eta_n \} = 4\eta_n = 4n^{-3}.
\]
If $I_{n,j}\in\mathcal B_n$, then there exists $k\le n-M_n-1$ such that
$\norm{T^k \tilde{t}_{n,j}}_{\mathbb{Z}}<2\eta_n$.
The $2^n$ values ${T}^k\tilde{t}_{n,j}$
 consist of $2^{n-k}$ equally spaced points, each with multiplicity $2^k$, 
and the distinct locations are spaced by $2^{k-n}$, 
hence they are uniformly distributed on $[0,1]$.
Then, the number of indices $j$ satisfying 
$\norm{T^k \tilde{t}_{n,j}}_{\mathbb{Z}}<2\eta_n$ is at most
$$2\left\lceil\frac{2\eta_n}{2^{k-n}}\right\rceil 2^k
<(2\eta_n+2^{k-n})2^{n+1}<(2\eta_n+2^{-M_n})2^{n+1}<6n^{-3}2^n.$$
We obtain
\[
\mathcal{L}\Bigl(\bigcup_{I_{n,j}\in\mathcal B_n} I_{n,j}\Bigr)
\le (n-M_n)6n^{-3}
\le 6n^{-2}.
\]
\end{proof}

\begin{prop}\label{prop:A_estimate}
Take any $j\in\Lambda_n$. If $I_{n,j}\in\mathcal G_n$, then for any $t\in I_{n,j}$, 
\[
|A(t)-A(t_{n,j})| = O(n^{-3}).
\]
\end{prop}
\begin{proof}
Suppose $j\in\Lambda_n$ and $I_{n,j}\in\mathcal G_n$.
Take any $t\in I_{n,j}$. Set for $0\le k\le n-M_n-1$,
$$x_k = T^k t_{n,j},\quad y_k = T^k t,\quad z_k = T^k \tilde{t}_{n,j}.$$
We have
$$\dist(x_k,y_k), \dist(y_k,z_k) \le 2^{k-n} \le 2^{-M_n}\le n^{-6}.$$
By the goodness condition $\norm{z_k}_{\mathbb{Z}}\ge 2\eta_n = 2n^{-3}$, 
we have,
\begin{equation}\label{disty}
\norm{y_k}_{\mathbb{Z}}=\norm{T^kt}_{\mathbb{Z}}\ge n^{-3}.
\end{equation}

By the mean value theorem, there exists $u$ between $x_k$ and $y_k$ such that
$$|\varphi(y_k)-\varphi(x_k)| \le \dist(y_k,x_k) \cdot |\varphi'(u)|.$$
By arbitrry of $t$ and \eqref{disty}, we have $\norm{u}_{\mathbb{Z}}\ge n^{-3}$. 
Since the derivative of $\varphi$ away from integers is $\varphi'(x)=2\pi\cot(\pi x)$,
$$|\varphi'(u)| = 2\pi|\cot(\pi\norm{u}_{\mathbb{Z}})| 
\le \frac{2\pi}{\pi\norm{u}_{\mathbb{Z}}} \le 2n^3.$$
Then 
$$|\varphi(y_k)-\varphi(x_k)| \le 2^{k-n}\cdot 2n^3.$$
Summing over $k=0,\dots,n-M_n-1$,
\[
|A(t)-A(t_{n,j})|
\le 2n^3\sum_{k=0}^{n-M_n-1} 2^{k-n}
= 2n^3\,2^{-n}2^{n-M_n}
\le 2n^3\,2^{-M_n}\le 2n^{-3}.
\]
\end{proof}

\paragraph{The second part $B$}

\begin{prop}\label{prop:B_estimate}
For any $j\in\Lambda_n$ and any $t\in I_{n,j}$,
\[
\sum_{k=n-M_n}^{n-1} \left( \varphi(T^k t) - \varphi(T^k t_{n,j})\right) < C\log n
\]
for some constant $C>0$.
\end{prop}
\begin{proof}
Take any $j\in\Lambda_n$, $t\in I_{n,j}$ and $k=n-M_n,\cdots,n-1$.
Set $x_k = T^k t_{n,j}$, $y_k = T^k t$.
Since $T^k I_{n,j}=[j/2^{n-k},(j+1)/2^{n-k})\mod 1$,
the infimum of $\|t\|_{\mathbb{Z}}$ for $t\in T^k I_{n,j}$ is equal to
$$\min\{\|j/2^{n-k}\|_{\mathbb{Z}}, \|(j+1)/2^{n-k}\|_{\mathbb{Z}}\},$$
which is equal to $l/2^{n-k}$, for some $0\le l\le 2^{n-k-1}$.
If $l=0$, then $\|y_k\|_{\mathbb{Z}}\le 2^{k-n}$, and $\|x_k\|_{\mathbb{Z}}\ge a2^{k-n}$, 
we have
$$\frac{\|y_k\|_{\mathbb{Z}}}{\|x_k\|_{\mathbb{Z}}}\le a^{-1}=10.$$
If $l>0$, then $\|y_k\|_{\mathbb{Z}}\le (l+1)2^{k-n}$,
and $\|x_k\|_{\mathbb{Z}}\ge (l+a)2^{k-n}$, we have
$$\frac{\|y_k\|_{\mathbb{Z}}}{\|x_k\|_{\mathbb{Z}}}\le2.$$

Then we have
\[
\left|\frac{\sin\pi y_k}{\sin\pi x_k}\right|=
\frac{\sin\pi \|y_k\|_{\mathbb{Z}}}{\sin\pi \|x_k\|_{\mathbb{Z}}}\le 2/a.
\]

Consequently,
\[
\Delta_k:= \varphi(y_k) - \varphi(x_k) = 2\log\left|\frac{\sin\pi y_k}{\sin\pi x_k}\right|
\le 2\log 2/a.
\]

The sum over the $M_n$ terms is at most
\[
\sum_{k=n-M_n}^{n-1} \Delta_k \le M_n 2\log 2/a = O(\log n).
\]
\end{proof}

We continue the proof of Lemma \ref{discrete}.
Take $j\in \Lambda_n$
such that
$I_{n,j}\in \mathcal G_n.$
By the definition of \(\Lambda_n\), there exists
\[
t_{n,j}\in \Gamma_n\cap \widetilde I_{n,j},
\]
and hence
\[
S_n\varphi(t_{n,j})
\leq
-\log^2 n-\log 80.
\]

For any \(t\in I_{n,j}\) and $n$ sufficiently large, Propositions \ref{prop:A_estimate} and \ref{prop:B_estimate} give
\[
S_n\varphi(t)\le S_n\varphi(t_{n,j})+
O(n^{-3})+O(\log n)\le -\log80.
\] Therefore,
$I_{n,j}\in Y_n.$

Since $\#\Lambda_n\geq \frac{2^n}{4}$, by Proposition \ref{prop:bad_measure}
\[\liminf_{n\to\infty} \frac{\#Y_n}{2^n}\geq
\liminf_{n\to\infty}\frac{\#\Lambda_n-n^{-1}\,2^n}{2^n}
\ge \frac14. \]
This completes the proof of Lemma \ref{discrete}.%\hfill $\square$

\section{cut-off intervals with length control}

We summarize the proof of Lemma 1 in \cite{Last94} 
in the following form for our purpose.

\begin{lemm}
For $n\ge 2$, let $x_n(t)$ be a trace polynomial, 
and $I_n$ be a cut-off interval of level $n$ for some fixed $K>0$.
Then there exist constants $C_3,C_4$ independent of $n, x_n$ and $K$, such that
\begin{equation*}
  \frac{C_3}{|x_{n}'(t_0)|}<|I_n|<\frac{C_4}{|x_{n}'(t_0)|},
\end{equation*}
where $t_0$ is the root of $x_n$ in $I_n$.
\end{lemm}

\begin{lemm}
    \label{length}
Fix $K>K_1$. If $I_0$ is a cut-off interval of level $n$ 
with $d(n,I_0)<\min\{\delta_1,\delta_3\}$,
then there exists a collection of intervals
$\mathcal{I}_{0,K}^{(3)}(I_0)\subset\mathcal{I}_{0,K}^{(2)}(I_0)$,
with $\#\mathcal{I}_{0,K}^{(3)}(I_0)>2^{K}/13$, 
and for each $I\in \mathcal{I}_{0,K}^{(3)}(I_0)$,
\begin{equation}
|I|>C|I_0|2^{-K},
\end{equation}
where the constant $C$ is an absolute constant independent of $K$.
\end{lemm}

\begin{proof}

Let $\varepsilon(=2^{-20K})$ and $\xi$ be as defined in the proof of Lemma \ref{lemma3}. 
We have $\xi(\pi/2)$ be the root of $x_n$ in $I_0$. 
$I_0$ is contained in a monotone branch interval $\tilde{I}$ for $x_n$.
There is a unique inflection point $\bar{t}$ of $x_n$ in $\tilde{I}$, 
that is, $x_n''(\bar{t})=0$.
Without loss of generality, assume that $x_n$ is monotone increasing on $I_0$.
The monotone decreasing case is analogous.

If $x_n(\bar{t})>x_n(\xi(\pi/2))=0$, we keep the intervals
$I\in\mathcal{I}_{0,K}^{(2)}(I_0)$ such that
$x_n(t)\in[2\cos{199\pi/200},0]$ for all $t\in I$, 
otherwise we keep the intervals $I\in\mathcal{I}_{0,K}^{(2)}(I_0)$ such that
$x_n(t)\in[0,2\cos{\pi/200}]$ for all $t\in I$. 
In both cases, for $t$ in selected intervals, 
$$|x_n'(t)|<|x_n'(\xi(\pi/2))|.$$ 

Denote the collection of intervals taken above by
 $\mathcal{I}_{0,K}^{(3)}(I_0)$.
 By Corollary \ref{1924}, $\mathcal{I}_{0,K}^{(2)}(I_0)$ is
symmetric. Hence, we have $\#\mathcal{I}_{0,K}^{(3)}(I_0)>2^K/13$.

Notice that by the chain rule,
\begin{equation*}
(f^n)'(x)=\prod_{i=0}^{n-1}f'(f^{i}(x))=\prod_{i=0}^{n-1}2(f^{i}(x)).
\end{equation*}
Substituting $x=2\cos{\theta}$ into the above equation, we obtain for $\pi/200<\theta<199\pi/200$,
\begin{equation*}
|(f^n)'(x)|=2^n\prod_{i=0}^{n-1}2|\cos{2^i\theta}|=\frac{2^n|\sin{2^n\theta}|}{\sin{\theta}}
<100\cdot 2^n,
\end{equation*}
since $1/\sin{\theta}<100$.
Take $I\in \mathcal{I}_{0,K}^{(3)}(I_0)$.
Let $t_1,t_2$ be two points in $I$ such that
$x_{n+K}(t_1)=-1,x_{n+K}(t_2)=1$.
Then
\begin{equation*}
    \begin{aligned}
   2&=x_{n+K}(t_2)-x_{n+K}(t_1)\\
   &\le f^{K}(x_n(t_2))-f^{K}(x_n(t_1))+2\varepsilon\\
   &=\int_{t_1}^{t_2}(f^{K}(x_n(t)))'\mathrm{d}t+2\varepsilon\\
   &\le 100\cdot 2^K\int_{t_1}^{t_2}|x_n'(t)|\mathrm{d}t+2\varepsilon\\
   &\le 100\cdot 2^K|x_n'(\xi(\pi/2))||t_2-t_1|+2\varepsilon\\
   &\le 100\cdot 2^K|x_n'(\xi(\pi/2))||I|+2\varepsilon,
    \end{aligned}
\end{equation*}
where the first inequality is due to \eqref{err}.
Since $\varepsilon\ll1/2$, we obtain that 
\begin{equation*}
    \begin{aligned}
  |I|\ge \frac{2-2\varepsilon}{100\cdot 2^K|x_n'(\xi(\pi/2))|}
  \ge \frac{|I_0|}{100C_4\cdot 2^K}.
    \end{aligned}
\end{equation*}
\end{proof}
\section{Separate nested structure and Hausdorff dimension}

We recall the definition of separated nested structure(SNS) in \cite{LQY22}.
A sequence of sets of intervals 
$\mathcal{I}=\{\mathcal{I}_n\}_{n\ge0}$ is called an SNS if

i) $\mathcal{I}$ is separating: $\forall n\ge0$, 
$I\ne J\in\mathcal{I}_n$, $I^o\cap J^o=\emptyset$; 

ii) $\mathcal{I}$ is nested: $\forall n\ge0$, $I\in \mathcal{I}_{n+1}$,
there is a unique $J\in\mathcal{I}_n$ such that $I\subset J$;

iii) $\mathcal{I}$ is minimal: $\forall n\ge0$, any $J\in \mathcal{I}_{n}$,
there exists $I\in\mathcal{I}_{n+1}$ such that $I\subset J$.

The limit set of $\mathcal{I}$ is defined as
$$E(\mathcal{I}):=\bigcap_{n\ge1}\bigcup_{I\in \mathcal{I}_n}I.$$

For $K>K_1$, take $\delta_1,\delta_3$ 
determined by Lemma \ref{lemma3} and Corollary \ref{1924}.
Note that $\delta_1,\delta_3$ depend on $K$.
We define an SNS $\mathcal{I}^{(K)}$ as follows.

By Lemma \ref{smalldevi}, there exists a band $I_0$ of level $N_0$ such that
$d(N_0,I_0)<\min\{\delta_1,\delta_3\}$.
We define $\mathcal{I}_0^{(K)}=\{I_0\}$.

By Lemma \ref{length},
there exists a set of cut-off intervals $\mathcal{I}_{0,K}^{(3)}(I_0)$,
such that $\#\mathcal{I}_{0,K}^{(3)}(I_0)>2^{K}/13$, 
and for each $I\in \mathcal{I}_{0,K}^{(3)}(I_0)$,
$I$ is a constructive interval of 
$$I_0\cap\sigma_{N_0+1}\cap\cdots\cap \sigma_{N_0+K},$$ 
and satisfies $|I|>C|I_0|2^{-K}$, $d(N_0+K,I)<d(N_0,I_0)$.

We take any subset of $\mathcal{I}_{0,K}^{(3)}(I_0)$ with $2^{K-4}$ cut-off intervals
and denote it by $\mathcal{I}_1^{(K)}$.

Suppose $\mathcal{I}_n^{(K)}$ is defined for some $n\ge1$.
We construct $\mathcal{I}_{n+1}^{(K)}$.
For any cut-off interval $J\in\mathcal{I}_n^{(K)}$,
 Lemma \ref{length} gives a set of cut-off intervals $\mathcal{I}_{0,K}^{(3)}(J)$,
such that $\#\mathcal{I}_{0,K}^{(3)}(J)>2^{K}/13$, 
and for each $I\in \mathcal{I}_{0,K}^{(3)}(J)$,
$I$ is a constructive interval of 
$$J\cap\sigma_{N_0+nK+1}\cap\cdots\cap \sigma_{N_0+(n+1)K},$$ 
and satisfies $|I|>C|J|2^{-K}$, $d(N_0+(n+1)K,I)<d(N_0,I_0)$.

We choose any  $2^{K-4}$ cut-off intervals from $\mathcal{I}_{0,K}^{(3)}(J)$ and put them
into $\mathcal{I}_{n+1}^{(K)}$.
Note that $\#\mathcal{I}_{n+1}^{(K)}=2^{(K-4)(n+1)}$.

By induction, we define the SNS $\mathcal{I}^{(K)}$.
It is clear that 
$$E(\mathcal{I}^{(K)})\subset 
\bigcap_{n\ge N_0}\sigma_n\subset
\bigcap_{n\ge N_0}(\sigma_{n}\cup\sigma_{n+1})=
\bigcap_{n\ge1}(\sigma_{n}\cup\sigma_{n+1})
=\sigma(H_{\mathrm{tm},\lambda}).$$

We recall  Proposition 5.7 in \cite{LQY22}:
\begin{prop}
\label{1944}
Let $\mathcal{I}=(\mathcal{I}_n)_{n\geq 1}$ be a SNS. 
Assume that it satisfies 

 (i) there exist $w\in(0,1)$ and $C>0$ such that, for any $n>0$ and $I\in\mathcal{I}_n$,
    $$
    |I|\geq Cw^n;
    $$
    
 (ii) there exists $C'\geq 1$ such that, for any $n,k> 0$ and $I,I'\in\mathcal{I}_n$,
    $$
    \frac{
    \#\{J\in\mathcal{I}_{n+k}:J\subset I\}
    }{
    \#\{J\in\mathcal{I}_{n+k}:J\subset I'\}
    }
    \leq C'.
    $$
Then
$$
\dim_H E(\mathcal{I})
\geq
\liminf_{n\to\infty}
\frac{\log \#\mathcal{I}_n}{-n\log w}.
$$
\end{prop}

\begin{thm}
$\lim\limits_{K\to\infty}\mathrm{dim}_{H}(E(\mathcal{I}^{(K)}))=1.$
\end{thm}    

\begin{proof}
  First we verify the condition of Proposition \ref{1944} for 
$\mathcal{I}^{(K)}$.
By our construction, for any $n,k\geq 0$ and $I,I'\in\mathcal{I}_n^{(K)}$,
 $$
    \frac{
    \#\{J\in\mathcal{I}_{n+k}^{(K)}:J\subset I\}
    }{
    \#\{J\in\mathcal{I}_{n+k}^{(K)}:J\subset I'\}
    }
    =1.
    $$
Increasing $K_1$ if necessary, assume $C2^{-K}<1$ for every $K>K_1$. Take $\omega=C2^{-K}$. We have, for $I\in\mathcal{I}_n^{(K)}$,
\begin{equation*}
|I|>|I_0|\omega^{n}.
\end{equation*}
Thus Proposition \ref{1944} gives
\begin{equation*}
    \mathrm{dim}_H(E(\mathcal{I}^{(K)}))\ge
    \liminf_{n\to\infty}\frac{\log\# \mathcal{I}_{n}^{(K)}}
    {-n\log \omega}
    =\frac{(K-4)\log2}{K\log2-\log{C}}.
    \end{equation*}
Then the limit is $1$ as $K\to\infty$.
\end{proof}

\begin{proof}[Proof of Theorem \ref{thm:main}]
Since for any $K>K_1$, $E(\mathcal{I}^{(K)})$ is a subset of the spectrum 
$\sigma(H_{\mathrm{tm},\lambda})$,
we have $\mathrm{dim}_{H}(\sigma(H_{\mathrm{tm},\lambda}))=1$ for any $\lambda\ne0$.
\end{proof}

\smallskip
\noindent\textbf{Acknowledgements}
We thank Professor Wen Zhiying for helpful discussions. 
This work is dedicated to him on his 80th birthday.
We thank GPT-5.5-Plus for its suggestions on the application of the central limit theorem.

\end{document}